\documentclass[reqno]{amsart}

\usepackage[T1]{fontenc}
\usepackage[american]{babel}
\usepackage{lmodern}
\usepackage{microtype}
\usepackage{amsmath,amssymb,amsthm}
\usepackage[hidelinks]{hyperref}
\hypersetup{
  pdftitle={Prime-Detecting Identities from Dirichlet Inversion and Jordan Totients},
  pdfauthor={Zhichen Liu},
  pdfkeywords={arithmetic functions, Dirichlet inverse, Jordan totient, primes}
}

\newcommand{\N}{\mathbb{N}}
\newcommand{\eps}{\varepsilon}
\newcommand{\one}{\mathbf{1}}
\newcommand{\Dinv}[1]{#1^{(-1)}}

\newtheorem{theorem}{Theorem}
\newtheorem{proposition}[theorem]{Proposition}

\theoremstyle{remark}

\title[Prime-Detecting Identities from Dirichlet Inversion]
{Prime-Detecting Identities from Dirichlet Inversion and Jordan Totients}
\author{Zhichen Liu}
\date{}


\begin{document}

\begin{abstract}
For each integer $k\geq1$, let $\sigma_k(n)=\sum_{d\mid n}d^k$, let
$\Dinv{\sigma_k}$ denote its Dirichlet inverse, and let $J_k$ denote the
$k$th Jordan totient function.  Using established prime-power values of
$\Dinv{\sigma_k}$, we first obtain the elementary criterion
$\Dinv{\sigma_k}(n)+J_k(n)+2=0$ precisely when $n$ is prime.  At
$k=1$, this involves Euler's totient $J_1=\varphi$ and provides a
prime-only comparison for the next equation.  Our main
result classifies the zeros of
$\Dinv{\sigma_k}(n)+\mu(n)+J_k(n)+3$: for $k=1$, they are the primes and
$18$; for $k\geq2$, they are only the primes.  The proof separates
integers by prime-exponent pattern.  Sign and divisibility arguments
exclude composite zeros outside the cube-free nonsquarefree case, where
an additional divisibility condition isolates $18$ at $k=1$ and
precludes composite zeros for higher $k$.  Standard convolution and
Lambert-series identities provide the arithmetic framework, and an
intermediate-divisor relation explains the exceptional zero.  The
resulting criteria are structural characterizations rather than
efficient primality tests, because direct evaluation of the
multiplicative formulas generally presupposes a factorization of $n$.
\end{abstract}

\maketitle

\section{Introduction}

Prime characterizations built from arithmetic functions are most useful
when the functions involved share a structural origin.  The identity studied
here arises from two standard transforms of the completely multiplicative
function $I_k(n)=n^k$.  Under Dirichlet convolution,
\begin{equation}
  \sigma_k=\one*I_k,
  \qquad
  J_k=\mu*I_k,
\end{equation}
where $\one(n)=1$ and $\mu$ is the M\"obius function.  Thus $\sigma_k$ is the
ordinary divisor transform of $I_k$, while $J_k$ is its M\"obius transform.

At a prime $p$, these related functions satisfy
\begin{equation}
  \Dinv{\sigma_k}(p)=-(1+p^k),
  \qquad J_k(p)=p^k-1,
  \qquad \mu(p)=-1.
\end{equation}
The first two values sum to $-2$, independently of $p$ and $k$.  Define
\begin{equation}\label{eq:H-definition}
  H_k(n):=\Dinv{\sigma_k}(n)+J_k(n)+2.
\end{equation}
Throughout, $\N=\{1,2,3,\ldots\}$.  The first result shows that this
prime-level cancellation has no composite solutions.

\begin{theorem}\label{thm:two-function}
For all $k,n\in\N$,
\begin{equation}\label{eq:two-function-criterion}
  \Dinv{\sigma_k}(n)+J_k(n)+2=0
  \quad\Longleftrightarrow\quad
  n\text{ is prime}.
\end{equation}
In particular, with $\sigma=\sigma_1$ and $J_1=\varphi$,
\begin{equation}\label{eq:two-function-classical}
  \Dinv{\sigma}(n)+\varphi(n)+2=0
  \quad\Longleftrightarrow\quad
  n\text{ is prime}.
\end{equation}
\end{theorem}

Adding $\mu(p)=-1$ makes the three prime-level values sum to $-3$.
This motivates the second function
\begin{equation}\label{eq:F-definition}
  F_k(n):=\Dinv{\sigma_k}(n)+\mu(n)+J_k(n)+3.
\end{equation}
Unlike $H_k$, this function has one composite zero at $k=1$.

\begin{theorem}\label{thm:main}
Let $k,n\in\N$.  Then
\begin{equation}
  F_k(n)=0
  \quad\Longleftrightarrow\quad
  \begin{cases}
    n\text{ is prime or }n=18, & k=1,\\
    n\text{ is prime}, & k\geq2.
  \end{cases}
\end{equation}
\end{theorem}

The exceptional value $18$ is the only composite that survives a
divisibility condition by $3$; that condition becomes impossible for
$k\geq2$.  Before proving the two classifications, we place them among
classical prime criteria and the literature on inverse arithmetic functions.

\section{Related work and contribution}

Several classical arithmetic functions characterize primes by their values
at a single integer.  For $n>1$, one has
\begin{equation}
  \varphi(n)=n-1\quad\Longleftrightarrow\quad n\text{ is prime},
\end{equation}
while Lehmer's celebrated problem asks whether the stronger divisibility
$\varphi(n)\mid n-1$ can hold for a composite $n$ \cite{lehmer}.
The Jordan analogue is $J_k(n)=n^k-1$ at primes.  Subbarao and Siva Rama
Prasad studied the corresponding divisibility condition
$J_k(n)\mid n^k-1$ and related analogues of Lehmer's totient problem
\cite{subbarao-prasad}.  Likewise,
$\sigma_k(n)=n^k+1$ characterizes primes because a composite integer has a
proper divisor other than $1$.

These familiar criteria use the values of $\sigma_k$ or $J_k$ themselves.
The present work instead uses the inverse of $\sigma_k$ in the ring of
arithmetic functions.  Cashwell and Everett developed the ring-theoretic
setting for arithmetic functions under Dirichlet convolution
\cite{cashwell}; McCarthy gives a systematic account of multiplicative
functions and convolution \cite{mccarthy}.  General
explicit expressions for Dirichlet inverses were studied by Haukkanen
\cite{haukkanen}.  Brown gave an expository treatment of inverse arithmetic
functions and, for the ordinary divisor sum, recorded
\begin{equation}
  \Dinv{\sigma_1}(p)=-(p+1),\qquad
  \Dinv{\sigma_1}(p^2)=p,\qquad
  \Dinv{\sigma_1}(p^a)=0\quad(a\geq3)
\end{equation}
\cite{brown}; the same sequence is catalogued as OEIS A046692
\cite{oeis}.  The corresponding generalized formulas also occur in analytic
work on $L$-functions \cite[eqs.~(4.7)--(4.8)]{zacharias}, and follow
immediately from $\sigma_k=\one*I_k$.

The inverse formulas just cited are established results.  The
two-function criterion in Theorem~\ref{thm:two-function} is an elementary
application of those formulas and provides a comparison for the
three-function equation.  Our focus is the exact zero set of
\eqref{eq:F-definition}, which adds the M\"obius value and raises the
constant from $2$ to $3$.  We did not find this mixed zero-set
classification in the cited treatments of inverse arithmetic functions;
this observation is limited to those sources and does not establish
priority over all earlier work.  The companion $n^k$ criterion below is
also a direct consequence of the known local values and is included to
clarify the role of intermediate divisors.

The criteria are structural rather than algorithmic.  Direct evaluation
of the multiplicative formulas generally presupposes a prime factorization,
so they are not proposed as faster primality tests.  Their significance is
that they link divisor summation, M\"obius transformation, and inversion in
the Dirichlet ring.  We first derive the constituent functions from their
Lambert-series identities and organize their local values by prime-exponent
pattern.  We then prove Theorem~\ref{thm:two-function} before
Theorem~\ref{thm:main}, followed by the $n^k$ criterion and the exact
intermediate-divisor relation between the three-function and $n^k$
criteria.

\section{Lambert-series derivation}\label{sec:generating}

For arithmetic functions $f$ and $g$, their Dirichlet convolution is
\begin{equation}\label{eq:dirichlet-convolution-definition}
  (f*g)(n)=\sum_{d\mid n}f(d)g(n/d).
\end{equation}
Its identity is the function $\eps$, defined by $\eps(1)=1$ and
$\eps(n)=0$ for $n>1$.  If $f(1)\neq0$, its Dirichlet inverse $\Dinv f$ is
determined by $f*\Dinv f=\eps$.  Throughout this article, $\Dinv f$ denotes
this convolution inverse, not a pointwise reciprocal.  Standard background
may be found in \cite[Chapter~2]{apostol} and \cite[Chapter~27]{dlmf}.

All identities in this section are formal power-series identities, so every
coefficient comparison is finite and requires no analytic convergence
assumption.  For any arithmetic function $a$,
\begin{equation}\label{eq:general-lambert}
  \sum_{r=1}^{\infty}\sum_{m=1}^{\infty}a(r)x^{rm}
  =\sum_{n=1}^{\infty}\left(\sum_{d\mid n}a(d)\right)x^n.
\end{equation}
Indeed, the coefficient of $x^n$ on the left comes from precisely the
pairs $(r,m)$ for which $rm=n$, equivalently from the positive divisors
$r$ of $n$.

Taking $a=\mu$ in \eqref{eq:general-lambert} gives the first generating
identity,
\begin{equation}\label{eq:mobius-generating}
  \sum_{r=1}^{\infty}\sum_{m=1}^{\infty}\mu(r)x^{rm}=x.
\end{equation}
Coefficient comparison yields
\begin{equation}\label{eq:mobius-divisor}
  \sum_{d\mid n}\mu(d)=\eps(n),
  \qquad
  \eps(n)=
  \begin{cases}
    1,&n=1,\\
    0,&n>1.
  \end{cases}
\end{equation}
In particular, the coefficients are determined recursively by
\begin{equation}\label{eq:mobius-recurrence}
  \mu(1)=1,
  \qquad
  \mu(n)=-\sum_{\substack{d\mid n\\d<n}}\mu(d)
  \quad(n>1).
\end{equation}
Multiplication by $n^k$ gives the arithmetic function
$(\mu I_k)(n)=n^k\mu(n)$ that appears below.

Taking $a=J_k$ in \eqref{eq:general-lambert} gives the second generating
identity,
\begin{equation}\label{eq:jordan-generating}
  \sum_{r=1}^{\infty}\sum_{m=1}^{\infty}J_k(r)x^{rm}
  =\sum_{n=1}^{\infty}n^k x^n.
\end{equation}
Its coefficient identity is
\begin{equation}\label{eq:jordan-divisor}
  \sum_{d\mid n}J_k(d)=n^k,
  \qquad\text{or equivalently}\qquad
  \one*J_k=I_k.
\end{equation}
Convolving with $\mu=\Dinv\one$ gives the standard formula
\begin{equation}\label{eq:jordan-convolution}
  J_k=\mu*I_k.
\end{equation}
For $k=1$, this is $J_1=\varphi$.

The third generating identity is
\begin{equation}\label{eq:inverse-generating}
  \sum_{r=1}^{\infty}\sum_{m=1}^{\infty}
       \Dinv{\sigma_k}(r)x^{rm}
  =\sum_{n=1}^{\infty}n^k\mu(n)x^n.
\end{equation}
Coefficient comparison gives
\begin{equation}\label{eq:inverse-divisor-generating}
  \sum_{d\mid n}\Dinv{\sigma_k}(d)=n^k\mu(n),
  \qquad\text{that is,}\qquad
  \one*\Dinv{\sigma_k}=\mu I_k.
\end{equation}
Convolving with $\mu$ yields
\begin{equation}\label{eq:inverse-identification}
  \Dinv{\sigma_k}=\mu*(\mu I_k).
\end{equation}
Complete multiplicativity of $I_k$ gives
\begin{equation}\label{eq:Ik-inverse-verification}
  \begin{aligned}
    \bigl(I_k*(\mu I_k)\bigr)(n)
      &=\sum_{d\mid n}d^k\mu(n/d)(n/d)^k\\
      &=n^k\sum_{d\mid n}\mu(n/d)
       =n^k\eps(n)=\eps(n).
  \end{aligned}
\end{equation}
This is indeed the Dirichlet inverse of
$\sigma_k=\one*I_k$, because
\begin{equation}\label{eq:inverse-verification}
  (\one*I_k)*\bigl(\mu*(\mu I_k)\bigr)
  =(\one*\mu)*(I_k*\mu I_k)
  =\eps*\eps=\eps.
\end{equation}
The coefficients may equivalently be computed by the recurrence
\begin{equation}\label{eq:inverse-recurrence}
  \Dinv{\sigma_k}(1)=1,
  \qquad
  \Dinv{\sigma_k}(n)
  =n^k\mu(n)
   -\sum_{\substack{d\mid n\\d<n}}\Dinv{\sigma_k}(d)
  \quad(n>1).
\end{equation}
Thus the three Lambert series lead directly to the standard functions
$\mu$, $J_k$, and $\Dinv{\sigma_k}$ used in the two prime
characterizations.

\section{Local values and factorization cases}

We now extract the local values needed in the proof.  Expanding the
convolution in \eqref{eq:inverse-identification} gives
\begin{equation}\label{eq:inverse-divisor-sum}
  \Dinv{\sigma_k}(n)
  =\sum_{d\mid n}\mu(d)\mu(n/d)(n/d)^k
  =\sum_{d\mid n}d^k\mu(d)\mu(n/d),
\end{equation}
where the second sum follows from the first by replacing $d$ with the
complementary divisor $n/d$.

The prime-power values follow directly from the divisor identity.  Let $p$
be prime and set $q=p^k$.  From
\eqref{eq:inverse-divisor-generating},
\begin{equation}\label{eq:prime-power-partial-sum}
  \sum_{j=0}^{a}\Dinv{\sigma_k}(p^j)=q^a\mu(p^a).
\end{equation}
For $a\geq1$, subtracting the same identity with $a-1$ in place of $a$
gives
\begin{equation}\label{eq:local-subtraction}
  \Dinv{\sigma_k}(p^a)
  =q^a\mu(p^a)-q^{a-1}\mu(p^{a-1}).
\end{equation}
Since $\mu(p)=-1$ and $\mu(p^a)=0$ for $a\geq2$, this yields
\begin{equation}\label{eq:local-inverse}
  \Dinv{\sigma_k}(p^a)=
  \begin{cases}
    1,&a=0,\\
    -(q+1),&a=1,\\
    q,&a=2,\\
    0,&a\geq3.
  \end{cases}
\end{equation}
Explicitly,
$\Dinv{\sigma_k}(p)=q(-1)-1=-(q+1)$,
$\Dinv{\sigma_k}(p^2)=0-q(-1)=q$, and
$\Dinv{\sigma_k}(p^a)=0$ for $a\geq3$.
The Jordan totient has the familiar formulas
\begin{equation}\label{eq:jordan}
  J_k(n)=\sum_{d\mid n}\mu(n/d)d^k
  =n^k\prod_{p\mid n}(1-p^{-k}),
\end{equation}
so that
\begin{equation}\label{eq:local-jordan}
  J_k(p^a)=q^{a-1}(q-1),\qquad a\geq1.
\end{equation}

Equations \eqref{eq:inverse-identification} and
\eqref{eq:jordan-convolution} show that both $\Dinv{\sigma_k}$ and $J_k$
are multiplicative.  We now record the three factorization cases explicitly
so that the subsequent proof can be read case by case.

\subsection{The three prime-exponent patterns}

\begin{proposition}\label{prop:factorization-cases}
Let \(n>1\).  Exactly one of the following three cases occurs.

\smallskip
\noindent\emph{Case I: The integer \(n\) is squarefree.}  Write
\begin{equation}
  n=\prod_{j=1}^r p_j,
  \qquad q_j=p_j^k,
\end{equation}
with distinct primes \(p_j\).  Then
\begin{equation}\label{eq:case-squarefree}
  \Dinv{\sigma_k}(n)=(-1)^r\prod_{j=1}^r(q_j+1),
  \qquad
  \mu(n)=(-1)^r,
  \qquad
  J_k(n)=\prod_{j=1}^r(q_j-1).
\end{equation}

\smallskip
\noindent\emph{Case II: A prime cube divides \(n\).}  If \(p^3\mid n\)
for some prime \(p\), then
\begin{equation}\label{eq:case-cube}
  \Dinv{\sigma_k}(n)=0,
  \qquad
  \mu(n)=0,
  \qquad
  J_k(n)>0.
\end{equation}

\smallskip
\noindent\emph{Case III: The integer \(n\) is cube-free but not squarefree.}
There are unique squarefree integers \(a,b\) such that
\begin{equation}\label{eq:ab2-factorization}
  n=ab^2,
  \qquad
  \gcd(a,b)=1,
  \qquad
  b>1.
\end{equation}
The primes dividing \(a\) occur to the first power in \(n\), whereas those
dividing \(b\) occur to the second power.  Consequently,
\begin{equation}\label{eq:case-ab2}
  \Dinv{\sigma_k}(ab^2)=\mu(a)b^k\sigma_k(a),
  \qquad
  \mu(ab^2)=0,
  \qquad
  J_k(ab^2)=b^kJ_k(a)J_k(b).
\end{equation}
\end{proposition}

\begin{proof}
The three cases exhaust the possible prime exponents.  In Case I,
\eqref{eq:case-squarefree} follows by multiplying the exponent-one values in
\eqref{eq:local-inverse} and \eqref{eq:local-jordan}.  In Case II, a local
factor \(\Dinv{\sigma_k}(p^a)=0\) with \(a\geq3\) makes the full
multiplicative product zero, while the repeated prime factor also gives
\(\mu(n)=0\).  The product formula in \eqref{eq:jordan} gives
\(J_k(n)>0\).

For Case III, place the primes of exponent one into \(a\) and those of
exponent two into \(b\).  The resulting \(a,b\) are squarefree, coprime, and
unique.  Multiplying the local values gives
\begin{equation}
  \Dinv{\sigma_k}(ab^2)
  =\prod_{p\mid a}\bigl(-(1+p^k)\bigr)
     \prod_{p\mid b}p^k
  =\mu(a)b^k\sigma_k(a),
\end{equation}
and
\begin{equation}
  J_k(ab^2)
  =\prod_{p\mid a}(p^k-1)
     \prod_{p\mid b}p^k(p^k-1)
  =b^kJ_k(a)J_k(b).
\end{equation}
Since $b>1$, the integer $ab^2$ has a repeated prime factor, so
$\mu(ab^2)=0$.  This proves \eqref{eq:case-ab2} and completes the proof of
the proposition.
\end{proof}

\smallskip
\noindent\textit{Proof organization.}
Theorem~\ref{thm:two-function} is proved first using the three
factorization cases.  For Theorem~\ref{thm:main}, Part I treats $k=1$,
including the
odd/even subcases for the squarefree factor $a$ and, when $a$ is even, the
possibility that its remaining odd factor equals $1$.  After that
classification is complete, Part II lists the value of $F_k(n)$ in each of
the same three cases for $k\geq2$ and excludes every composite possibility.

\section{The two-function criterion}

\begin{proof}[Proof of Theorem~\ref{thm:two-function}]
At $n=1$, the left side of \eqref{eq:two-function-criterion} is $4$.
For a prime $p$, \eqref{eq:local-inverse} and
\eqref{eq:local-jordan} give
$\Dinv{\sigma_k}(p)+J_k(p)+2=-(1+p^k)+(p^k-1)+2=0$.

Now let $n$ be composite.  If a prime cube divides $n$, then
$\Dinv{\sigma_k}(n)=0$ and $J_k(n)>0$ by \eqref{eq:case-cube}, so the
left side is positive.  If $n$ is squarefree, write
$n=\prod_{j=1}^r p_j$, where $r\geq2$, and set $q_j=p_j^k$.
For even $r$, both terms $\Dinv{\sigma_k}(n)$ and $J_k(n)$ are positive.
For odd $r$, we have $r\geq3$, and the difference-of-products estimate
gives
\begin{equation}\label{eq:two-function-squarefree-bound}
  \prod_{j=1}^r(q_j+1)-\prod_{j=1}^r(q_j-1)
  \geq 2\prod_{j=2}^r q_j>2.
\end{equation}
Thus the left side of \eqref{eq:two-function-criterion} is negative,
not zero.

It remains to consider a cube-free, nonsquarefree integer.  Write
$n=ab^2$ as in \eqref{eq:ab2-factorization}.  By \eqref{eq:case-ab2},
\begin{equation}\label{eq:two-function-ab2}
  \Dinv{\sigma_k}(n)+J_k(n)+2
  =b^k\bigl(\mu(a)\sigma_k(a)+J_k(a)J_k(b)\bigr)+2.
\end{equation}
If $k\geq2$, then $b^k\geq4$, so the integer on the right of
\eqref{eq:two-function-ab2} cannot be zero.  If $k=1$, a zero would
require $b\mid2$, hence $b=2$ and $a$ is odd.  Since $J_1(2)=1$,
\eqref{eq:two-function-ab2} would then require
\begin{equation}\label{eq:two-function-odd-a}
  \mu(a)\sigma_1(a)+\varphi(a)=-1.
\end{equation}
For $a=1$, the left side is $2$.  For odd $a>1$, both
$\sigma_1(a)=\prod_{p\mid a}(p+1)$ and
$\varphi(a)=\prod_{p\mid a}(p-1)$ are even, so the left side is even.
Either way, \eqref{eq:two-function-odd-a} is impossible.  No composite
integer satisfies \eqref{eq:two-function-criterion}.
\end{proof}

\section{The three-function zero set}

\begin{proof}[Proof of Theorem~\ref{thm:main}]
At $n=1$,
\begin{equation}\label{eq:F-at-one}
  F_k(1)=1+1+1+3=6.
\end{equation}
For every prime $p$ and every $k\geq1$,
\begin{equation}\label{eq:F-at-prime}
  F_k(p)=-(1+p^k)-1+(p^k-1)+3=0.
\end{equation}
It remains to classify the composite zeros.  We first prove the statement
for $k=1$ and then treat $k\geq2$ separately.

\medskip
\noindent\emph{Part I: The case $k=1$.}
Assume that $F_1(n)=0$.  By \eqref{eq:F-at-one}, we have $n>1$.

\smallskip
\noindent\emph{Case I: The integer $n$ is squarefree.}
For squarefree $n$,
\begin{equation}\label{eq:k-one-squarefree-components}
  \Dinv{\sigma_1}(n)=\mu(n)\sigma_1(n),
  \qquad
  J_1(n)=\varphi(n),
\end{equation}
Substituting \eqref{eq:k-one-squarefree-components} into
\eqref{eq:F-definition} gives
\begin{equation}\label{eq:F-one-squarefree}
  F_1(n)=\mu(n)\bigl(\sigma_1(n)+1\bigr)+\varphi(n)+3.
\end{equation}

\noindent\emph{Subcase I(a): $\mu(n)=1$.}
Equation \eqref{eq:F-one-squarefree} gives
\begin{equation}
  F_1(n)=\sigma_1(n)+\varphi(n)+4>0.
\end{equation}

\smallskip
\noindent\emph{Subcase I(b): $\mu(n)=-1$.}
Now $F_1(n)=0$ is equivalent to
\begin{equation}\label{eq:k-one-squarefree}
  \sigma_1(n)-\varphi(n)=2.
\end{equation}
If $n$ has one prime factor, then $n$ is prime.  If $n$ is composite, it has
a divisor $d$ with $1<d<n$, and
\begin{equation}
  \sigma_1(n)-\varphi(n)
  \geq(1+d+n)-(n-1)=d+2>2,
\end{equation}
contradicting \eqref{eq:k-one-squarefree}.  Thus the squarefree solutions are
exactly the primes.

\smallskip
\noindent\emph{Case II: A prime cube divides $n$.}
If $p^3\mid n$ for some prime $p$, then
\begin{equation}\label{eq:k-one-cube-components}
  \Dinv{\sigma_1}(n)=\mu(n)=0,
\end{equation}
so \eqref{eq:k-one-cube-components} and \eqref{eq:F-definition} give
\begin{equation}
  F_1(n)=\varphi(n)+3>0.
\end{equation}
This case contains no solutions.

\smallskip
\noindent\emph{Case III: The integer $n$ is cube-free but not squarefree.}
Write
\begin{equation}\label{eq:k-one-ab2-factorization}
  n=ab^2,
  \qquad
  a,b\text{ squarefree},
  \qquad
  \gcd(a,b)=1,
  \qquad
  b>1.
\end{equation}
For the factorization \eqref{eq:k-one-ab2-factorization},
equation \eqref{eq:case-ab2} with $k=1$ gives
\begin{equation}\label{eq:k-one-ab2-components}
  \Dinv{\sigma_1}(ab^2)=\mu(a)b\sigma_1(a),
  \qquad
  \varphi(ab^2)=b\varphi(a)\varphi(b).
\end{equation}
Together with $\mu(ab^2)=0$, equation
\eqref{eq:k-one-ab2-components} gives
\begin{equation}\label{eq:k-one-b-divisibility}
  \begin{aligned}
    F_1(n)=0
    &\Longleftrightarrow
    b\bigl(\mu(a)\sigma_1(a)+\varphi(a)\varphi(b)\bigr)+3=0\\
    &\Longleftrightarrow
    b\bigl(\mu(a)\sigma_1(a)+\varphi(a)\varphi(b)\bigr)=-3.
  \end{aligned}
\end{equation}
Therefore $b\mid3$.  Because $b>1$, we have $b=3$, and
\eqref{eq:k-one-b-divisibility} becomes
\begin{equation}\label{eq:k-one-a}
  \mu(a)\sigma_1(a)+2\varphi(a)=-1.
\end{equation}
Positivity forces $\mu(a)=-1$, and hence $a>1$.

\noindent\emph{Subcase III(a): $a$ is odd.}
Then
\begin{equation}\label{eq:k-one-odd-sigma}
  \sigma_1(a)=\prod_{p\mid a}(p+1).
\end{equation}
By \eqref{eq:k-one-odd-sigma}, the left side of \eqref{eq:k-one-a} is even,
whereas its right
side is odd.  This is impossible.

\smallskip
\noindent\emph{Subcase III(b): $a$ is even.}
Write $a=2c$, where $c$ is odd and squarefree.  Since
\begin{equation}\label{eq:k-one-even-components}
  \sigma_1(a)=3\sigma_1(c),
  \qquad
  \varphi(a)=\varphi(c),
\end{equation}
substitution of \eqref{eq:k-one-even-components} into
\eqref{eq:k-one-a} reduces the equation to
\begin{equation}\label{eq:k-one-c}
  3\sigma_1(c)-2\varphi(c)=1.
\end{equation}

\noindent\emph{Subcase III(b1): $c>1$.}
Here
\begin{equation}
  3\sigma_1(c)-2\varphi(c)
  \geq3(c+1)-2(c-1)=c+5>1,
\end{equation}
contradicting \eqref{eq:k-one-c}.

\smallskip
\noindent\emph{Subcase III(b2): $c=1$.}
Then $a=2$ and
\begin{equation}\label{eq:k-one-exception}
  n=ab^2=2\cdot3^2=18.
\end{equation}
Moreover,
\begin{equation}\label{eq:k-one-exception-values}
  \begin{aligned}
    \Dinv{\sigma_1}(18)
      &=\Dinv{\sigma_1}(2)\Dinv{\sigma_1}(3^2)=(-3)(3)=-9,\\
    \mu(18)&=0,\qquad \varphi(18)=6,
  \end{aligned}
\end{equation}
so $F_1(18)=-9+0+6+3=0$.  This completes the classification for $k=1$:
the zeros are exactly the primes and $18$.

\medskip
\noindent\emph{Part II: The case $k\geq2$.}
Assume that $k\geq2$ and that a composite integer $n$ satisfies
$F_k(n)=0$.  We write the three components of $F_k(n)$ explicitly in each
factorization case.

\smallskip
\noindent\emph{Case I: The integer $n$ is squarefree.}
Write
\begin{equation}\label{eq:higher-squarefree-factorization}
  n=\prod_{j=1}^r p_j,
  \qquad q_j=p_j^k,
  \qquad r\geq2,
\end{equation}
where the primes $p_j$ are distinct.  Then
\begin{equation}\label{eq:higher-squarefree-components}
  \begin{aligned}
    \Dinv{\sigma_k}(n)&=(-1)^r\prod_{j=1}^r(q_j+1),\\
    \mu(n)&=(-1)^r,\\
    J_k(n)&=\prod_{j=1}^r(q_j-1).
  \end{aligned}
\end{equation}
Substituting \eqref{eq:higher-squarefree-components} into
\eqref{eq:F-definition} gives
\begin{equation}\label{eq:Fk-squarefree-form}
  F_k(n)
  =(-1)^r\left(\prod_{j=1}^r(q_j+1)+1\right)
    +\prod_{j=1}^r(q_j-1)+3.
\end{equation}
If $r$ is even, every term in \eqref{eq:Fk-squarefree-form} is positive.
If $r$ is odd and $n$ is composite, then $r\geq3$.  The equation
$F_k(n)=0$ would require
\begin{equation}\label{eq:higher-squarefree-equation}
  \prod_{j=1}^r(q_j+1)-\prod_{j=1}^r(q_j-1)=2.
\end{equation}
However,
\begin{align}
  \prod_{j=1}^r(q_j+1)-\prod_{j=1}^r(q_j-1)
  &={\left(\prod_{j=1}^r(q_j+1)-\prod_{j=1}^r q_j\right)}
    +{\left(\prod_{j=1}^r q_j-\prod_{j=1}^r(q_j-1)\right)}
    \label{eq:higher-product-decomposition}\\
  &\geq 2\prod_{j=2}^r q_j>2.
  \label{eq:higher-product-expansion}
\end{align}
In \eqref{eq:higher-product-decomposition}, each parenthesized difference
is at least
$\prod_{j=2}^r q_j$, and this product is greater than $1$ because
$r\geq3$ and $q_j=p_j^k\geq4$.
Thus no squarefree composite is a zero.

\smallskip
\noindent\emph{Case II: A prime cube divides $n$.}
If $p^3\mid n$ for some prime $p$, then
\begin{equation}\label{eq:higher-cube-components}
  \Dinv{\sigma_k}(n)=0,
  \qquad
  \mu(n)=0,
  \qquad
  J_k(n)=n^k\prod_{q\mid n}(1-q^{-k})>0.
\end{equation}
Therefore, \eqref{eq:higher-cube-components} and
\eqref{eq:F-definition} give
\begin{equation}\label{eq:Fk-cube-form}
  F_k(n)=J_k(n)+3>0.
\end{equation}
Thus this case contains no zeros.

\smallskip
\noindent\emph{Case III: The integer $n$ is cube-free but not squarefree.}
Write
\begin{equation}\label{eq:higher-ab2-factorization}
  n=ab^2,
  \qquad
  a,b\text{ squarefree},
  \qquad
  \gcd(a,b)=1,
  \qquad
  b>1.
\end{equation}
In \eqref{eq:higher-ab2-factorization}, the primes dividing $a$ occur to
the first power in $n$, while the primes dividing $b$ occur to the second
power.  Therefore,
\begin{equation}\label{eq:higher-ab2-components}
  \begin{aligned}
    \Dinv{\sigma_k}(n)&=\mu(a)b^k\sigma_k(a),\\
    \mu(n)&=0,\\
    J_k(n)&=b^kJ_k(a)J_k(b).
  \end{aligned}
\end{equation}
Therefore, \eqref{eq:higher-ab2-components} and
\eqref{eq:F-definition} give
\begin{equation}\label{eq:Fk-ab2-form}
  F_k(n)
  =b^k\bigl(\mu(a)\sigma_k(a)+J_k(a)J_k(b)\bigr)+3.
\end{equation}
If $F_k(n)=0$, then \eqref{eq:Fk-ab2-form} gives
\begin{equation}\label{eq:higher-ab2-divisibility}
  b^k\bigl(\mu(a)\sigma_k(a)+J_k(a)J_k(b)\bigr)=-3.
\end{equation}
We now expand both factors on the left.  Since $a$ and $b$ are squarefree,
write
\begin{equation}\label{eq:higher-a-b-prime-factorizations}
  a=\prod_{i=1}^{s}u_i,
  \qquad
  b=\prod_{\ell=1}^{t}v_\ell,
  \qquad s\geq0,\qquad t\geq1.
\end{equation}
In \eqref{eq:higher-a-b-prime-factorizations}, the $u_i$ and $v_\ell$ are
primes.  Set
$x_i=u_i^k$ and $y_\ell=v_\ell^k$.  With the convention that an empty
product equals $1$, the four terms in the bracket are
\begin{equation}\label{eq:higher-bracket-components}
  \begin{aligned}
    \mu(a)&=(-1)^s,
    &\sigma_k(a)&=\prod_{i=1}^{s}(x_i+1),\\
    J_k(a)&=\prod_{i=1}^{s}(x_i-1),
    &J_k(b)&=\prod_{\ell=1}^{t}(y_\ell-1).
  \end{aligned}
\end{equation}
Also,
\begin{equation}\label{eq:higher-b-lower-bound}
  b^k=\prod_{\ell=1}^{t}y_\ell\geq2^k\geq4.
\end{equation}
Set
\begin{equation}\label{eq:higher-A-B-definition}
  A:=\prod_{i=1}^{s}(x_i+1),
  \qquad
  B:=\prod_{i=1}^{s}(x_i-1)
       \prod_{\ell=1}^{t}(y_\ell-1).
\end{equation}
Using \eqref{eq:higher-bracket-components} and
\eqref{eq:higher-A-B-definition}, the bracketed factor becomes the integer
\begin{equation}\label{eq:higher-T-definition}
  T:=\mu(a)\sigma_k(a)+J_k(a)J_k(b)=(-1)^sA+B.
\end{equation}
By \eqref{eq:higher-T-definition}, if $s$ is even, then $T=A+B>0$, so
$b^kT>0$.  If $s$ is odd, then
$T=-A+B$.  When $B\geq A$, we have $T\geq0$; when $B<A$, integrality gives
$T\leq-1$, and hence \eqref{eq:higher-b-lower-bound} gives
$b^kT\leq-4$.  In no case can $b^kT=-3$, contradicting
\eqref{eq:higher-ab2-divisibility}.  Therefore no composite integer is a
zero when $k\geq2$.
By \eqref{eq:F-at-prime}, every prime is a zero, so the zeros for $k\geq2$
are exactly the primes.
This completes the proof.
\end{proof}

\section{An elementary companion criterion}

The established local values of $\Dinv{\sigma_k}$ also give a shorter
prime criterion.  We record it to compare with the exceptional zero of
the three-function equation.

\begin{theorem}\label{thm:companion}
For every $k\geq1$ and $n>1$,
\begin{equation}\label{eq:companion-criterion}
  \Dinv{\sigma_k}(n)+n^k=-1
  \quad\Longleftrightarrow\quad
  n\text{ is prime}.
\end{equation}
\end{theorem}

\begin{proof}
We first verify the equality for primes and then exclude every composite
integer through the following exhaustive cases.

\smallskip
\noindent\emph{Case I: The integer $n$ is prime.}
Write $n=p$, where $p$ is prime.  By \eqref{eq:local-inverse}, the two terms in
\eqref{eq:companion-criterion} and their sum are
\begin{equation}\label{eq:companion-prime-components}
  \begin{aligned}
    \Dinv{\sigma_k}(p)&=-(1+p^k),\\
    n^k&=p^k,\\
    \Dinv{\sigma_k}(p)+p^k
      &=-(1+p^k)+p^k=-1.
  \end{aligned}
\end{equation}
Thus every prime satisfies the equation on the left side of
\eqref{eq:companion-criterion}.

For the converse, let $n>1$ be composite.

\smallskip
\noindent\emph{Case II: A prime cube divides $n$.}
Write
\begin{equation}\label{eq:companion-cube-factorization}
  n=p^a m,
  \qquad
  a\geq3,
  \qquad
  \gcd(p,m)=1,
\end{equation}
where $p$ is prime and $m\in\N$.  By multiplicativity and
\eqref{eq:local-inverse}, the two terms and their sum are
\begin{equation}\label{eq:companion-cube-components}
  \begin{aligned}
    \Dinv{\sigma_k}(n)&=0,\\
    n^k&=p^{ak}m^k,\\
    \Dinv{\sigma_k}(n)+n^k
      &=0+p^{ak}m^k=p^{ak}m^k>0.
  \end{aligned}
\end{equation}
Therefore the sum cannot equal $-1$, so this case contains no solutions.

\smallskip
\noindent\emph{Case III: The integer $n$ is cube-free but not squarefree.}
Some prime $p$ occurs to exponent $2$.  Write
\begin{equation}\label{eq:companion-repeated-prime}
  n=p^2m,
  \qquad
  \gcd(p,m)=1.
\end{equation}
Here $p$ is prime and $m$ is cube-free.  By multiplicativity and
\eqref{eq:local-inverse}, the two terms and their sum are
\begin{equation}\label{eq:companion-repeated-prime-components}
  \begin{aligned}
    \Dinv{\sigma_k}(n)&=p^k\Dinv{\sigma_k}(m),\\
    n^k&=p^{2k}m^k,\\
    \Dinv{\sigma_k}(n)+n^k
      &=p^k\Dinv{\sigma_k}(m)+p^{2k}m^k\\
      &=p^k\bigl(\Dinv{\sigma_k}(m)+p^km^k\bigr).
  \end{aligned}
\end{equation}

The quantity in parentheses is an integer.  If it is zero, then
\eqref{eq:companion-repeated-prime-components} gives
$\Dinv{\sigma_k}(n)+n^k=0$.  If it is nonzero, then the sum is a nonzero
multiple of $p^k$, so its absolute value is at least $p^k\geq2$.  In either
case, the sum cannot equal $-1$.

\smallskip
\noindent\emph{Case IV: The integer $n$ is squarefree and composite.}
Write
\begin{equation}\label{eq:companion-squarefree-factorization}
  n=\prod_{j=1}^{r}p_j,
  \qquad
  q_j=p_j^k,
  \qquad
  r\geq2,
\end{equation}
where the primes $p_j$ are distinct.  By \eqref{eq:case-squarefree}, the two
terms and their sum are
\begin{equation}\label{eq:companion-squarefree-components}
  \begin{aligned}
    \Dinv{\sigma_k}(n)
      &=(-1)^r\prod_{j=1}^{r}(q_j+1),\\
    n^k&=\prod_{j=1}^{r}q_j,\\
    \Dinv{\sigma_k}(n)+n^k
      &=(-1)^r\prod_{j=1}^{r}(q_j+1)
        +\prod_{j=1}^{r}q_j.
  \end{aligned}
\end{equation}

\noindent\emph{Subcase IV(a): $r$ is even.}
Under this condition,
\begin{equation}\label{eq:companion-squarefree-even}
  \Dinv{\sigma_k}(n)+n^k
  =\prod_{j=1}^{r}(q_j+1)+\prod_{j=1}^{r}q_j>0.
\end{equation}
Therefore the sum cannot equal $-1$.

\smallskip
\noindent\emph{Subcase IV(b): $r$ is odd.}
Since $n$ is composite, this subcase has $r\geq3$.  Expanding the difference
of products below gives a sum of positive monomials.  It contains
$\prod_{j=2}^{r}q_j$ together with additional positive terms, so the
difference is greater than $q_2q_3$.  Therefore,
\begin{equation}\label{eq:companion-squarefree-odd}
  \Dinv{\sigma_k}(n)+n^k
  =-\prod_{j=1}^{r}(q_j+1)+\prod_{j=1}^{r}q_j
  =-\left(
      \prod_{j=1}^{r}(q_j+1)-\prod_{j=1}^{r}q_j
    \right)<-q_2q_3<-1.
\end{equation}
Therefore the sum cannot equal $-1$.

Cases II--IV exclude every composite integer, while Case I gives every prime.
This proves \eqref{eq:companion-criterion}.
\end{proof}

\medskip

\section{The intermediate-divisor relation}\label{sec:remainder}

There is also an exact relation between the three-function and $n^k$
criteria.  It shows how the
criterion changes when only the endpoint terms of the divisor sums are
kept.  For $k\geq1$ and $n>1$, define
\begin{equation}\label{eq:remainder-definition}
  R_k(n):=\sum_{\substack{d\mid n\\1<d<n}}
              \bigl(\mu(d)+J_k(d)\bigr).
\end{equation}
An empty sum in \eqref{eq:remainder-definition} is understood to be $0$.
Adding the divisor identities \eqref{eq:mobius-divisor} and
\eqref{eq:jordan-divisor}, and using $\eps(n)=0$ for $n>1$, gives
\begin{equation}\label{eq:power-decomposition}
  \begin{aligned}
    n^k
    &=\sum_{d\mid n}\bigl(\mu(d)+J_k(d)\bigr)\\
    &=2+\mu(n)+J_k(n)+R_k(n),
  \end{aligned}
\end{equation}
where the second line separates the contributions from $d=1$, $1<d<n$,
and $d=n$, and uses $\mu(1)+J_k(1)=2$.

\begin{proposition}\label{prop:remainder}
For every $k\geq1$ and $n>1$,
\begin{equation}\label{eq:remainder-bridge}
  F_k(n)+R_k(n)=\Dinv{\sigma_k}(n)+n^k+1.
\end{equation}
Consequently,
\begin{equation}\label{eq:corrected-criterion}
  F_k(n)+R_k(n)=0
  \quad\Longleftrightarrow\quad
  n\text{ is prime}.
\end{equation}
\end{proposition}

\begin{proof}
Using the definition of $F_k$ and then \eqref{eq:power-decomposition},
\begin{align}
  F_k(n)+R_k(n)
  &=\Dinv{\sigma_k}(n)+\mu(n)+J_k(n)+3+R_k(n)\\
  &=\Dinv{\sigma_k}(n)+n^k+1.
\end{align}
Equation \eqref{eq:corrected-criterion} is now exactly
Theorem~\ref{thm:companion} with $1$ added to both sides.
\end{proof}

For a prime $p$, there are no intermediate divisors and $R_k(p)=0$.  At the
exceptional composite for $k=1$, however,
\begin{equation}
  F_1(18)=0,
  \qquad
  R_1(18)=18-2-\mu(18)-\varphi(18)=10,
\end{equation}
and hence $F_1(18)+R_1(18)=10\neq0$.  Thus the remainder term explains
precisely why the exact divisor relation rejects the composite exception
whereas the shortened three-function identity does not.

\section{Conclusion}

For $k,n\in\N$ in the first two equivalences, and for $k\geq1$ and
$n>1$ in the third, the results proved in this note are
\begin{equation}
\begin{aligned}
  \Dinv{\sigma_k}(n)+J_k(n)+2=0
  &\quad\Longleftrightarrow\quad
  n\text{ is prime}\qquad(k\geq1),\\[3pt]
  \Dinv{\sigma_k}(n)+\mu(n)+J_k(n)+3=0
  &\quad\Longleftrightarrow\quad
  \begin{cases}
    n\text{ is prime or }n=18, & k=1,\\
    n\text{ is prime}, & k\geq2,
  \end{cases}\\[3pt]
  \Dinv{\sigma_k}(n)+n^k=-1
  &\quad\Longleftrightarrow\quad
  n\text{ is prime}\qquad(k\geq1,\ n>1).
\end{aligned}
\end{equation}
The three-function classification is the main result: its only
composite zero is $18$, and only when $k=1$.  The first and third
equivalences are elementary comparisons derived from the same known
local values of the inverse function.  The generating equations
\eqref{eq:mobius-generating}, \eqref{eq:jordan-generating}, and
\eqref{eq:inverse-generating} supply standard Lambert-series forms,
while \eqref{eq:remainder-bridge} shows why retaining intermediate
divisors removes the exceptional value $18$.  These are structural
characterizations, not computational substitutes for standard
primality tests.

\section*{Use of generative AI-assisted technologies}

OpenAI Codex assisted with language, references, and \LaTeX{} and PDF
preparation, as well as the formulation and proof check of the added
two-function criterion.  The main three-function classification is the
author's work.  The author will independently verify all mathematical
claims and sources before submission and takes full responsibility for
the manuscript.

\end{document}